\documentclass[12pt]{article}
\usepackage{amsmath,amssymb,amsthm,amsfonts,mathrsfs}
\usepackage{appendix}
\usepackage{array,enumitem,graphics,xcolor,yfonts,colonequals}
\usepackage{stmaryrd}
\usepackage[all,cmtip]{xy}
\usepackage{tikz-cd}
\usepackage[colorlinks,anchorcolor=blue,citecolor=blue,linkcolor=blue,urlcolor=blue,bookmarksopen=true]{hyperref}
\usepackage{comment}
\usepackage{mathtools}
\usepackage[margin=1in]{geometry}

\usepackage{titlesec}
\titleformat{\section}
  {\normalfont\large\bf}{\thesection}{1em}{}
\titleformat{\subsection}[runin]
  {\normalfont\normalsize\bf}{\thesubsection}{1em}{}
\titleformat{\subsubsection}[runin]
  {\normalfont\normalsize\bf}{\thesubsubsection}{1em}{}

\usepackage{fancyhdr}
\fancypagestyle{titlepage}
{
	\fancyhf{}

	\fancyfoot[l]{
	\href{https://mathscinet.ams.org/mathscinet/msc/msc2020.html}
        (Needs update!!!)
	}
}

\AtBeginDocument{%
	\def\MR#1{}
}

\newcommand{\bP}{\mathbb{P}}    
\newcommand{\bR}{\mathbb{R}}    
\newcommand{\bZ}{\mathbb{Z}}    

\newcommand{\sA}{\mathscr{A}}   
\newcommand{\sF}{\mathscr{F}}   
\newcommand{\sP}{\mathscr{P}}   
\newcommand{\sT}{\mathscr{T}}   

\newcommand{\Amp}{\mathrm{Amp}}         
\newcommand{\Arg}{\mathrm{Arg}}
\newcommand{\ch}{\mathrm{ch}}           
\newcommand{\Coh}{\mathrm{Coh}} 

\newcommand{\Db}{\mathrm{D^b}}  

\newcommand{\rank}{\mathrm{rank}}

\newcommand{\NS}{\operatorname{NS}}     

\newtheorem*{thm*}{Theorem}
\newtheorem*{prop*}{Proposition}
\newtheorem*{cor*}{Corollary}
\newtheorem*{ques*}{Question}
\newtheorem{thm}{Theorem}[section]

\newtheorem{conj}[thm]{Conjecture}

\numberwithin{equation}{section}

\theoremstyle{definition}
\newtheorem{defn}[thm]{Definition}
\newtheorem{eg}[thm]{Example}
\newtheorem{rmk}[thm]{Remark}

\begin{document}
\title{Notes on the deformed Hermitian--Yang--Mills equations and the large scaling limits of stability conditions}
\author{Yu-Wei Fan}
\date{}
\maketitle

\begin{abstract}
In this short note, we show that, assuming a conjecture of Arcara and Miles, a line bundle on a smooth complex projective surface admits a deformed Hermitian--Yang--Mills metric if and only if it is stable in the ``large scaling limit" with respect to a generic K\"ahler form. The same statement for toric surfaces was recently proved by Stoppa. The purpose of this note is to remark that this equivalence holds for arbitrary smooth projective surfaces.
\end{abstract}

\section{Introduction.}
The notion of stability conditions on triangulated categories was introduced by Bridgeland \cite{BriStab}, with inspiration from work in string theory. It is expected that on Fukaya-type categories, stable objects are given by special Lagrangian submanifolds. Under semi-flat mirror symmetry (i.e.~in the absence of quantum corrections), special Lagrangian sections are mirror to line bundles admitting deformed Hermitian--Yang--Mills (dHYM) metrics \cite{LYZ}. It is therefore natural to compare the Bridgeland stability of line bundles with the existence of dHYM metrics. In this note, we focus on the case of smooth complex projective surfaces, where both stability conditions are well-studied \cite{BriK3,ArcaraBertram,CollinsJacobYau}.

It is known that these two notions of stability are \emph{not} equivalent. For instance, there are examples of line bundles that are Bridgeland stable but do not admit solutions to the dHYM equation \cite[Page~31]{CollinsShi}. This is not unexpected; the central charge corresponding to the dHYM equation is only an approximation of the exact, quantum-corrected central charge at the large volume limit. Another crucial difference lies in \emph{scaling invariance}: dHYM stability is scaling-invariant, meaning that if a line bundle $L$ is dHYM-stable with respect to $\omega$, then $L^{\otimes k}$ is dHYM-stable with respect to $k\omega$ for any $k\geq1$. This property does not generally hold for the Bridgeland stability of line bundles (see Example~\ref{eg:not-scaling-invariant}).

To bridge this conceptual gap, Stoppa \cite{Stoppa2505,Stoppa2508} recently introduced a notion that incorporates this scaling invariance into Bridgeland stability conditions, which we will refer to as the ``\emph{large scaling limit}'' (termed the \emph{large volume, large Lagrangian limit} in \cite{Stoppa2505,Stoppa2508}). It is proved that, assuming a conjecture by Arcara and Miles (see Conjecture~\ref{conj:ArcaraMiles}), a line bundle on a toric surface is dHYM-stable if and only if it is stable in the large scaling limit for a generic K\"ahler class \cite[Proposition~6.2]{Stoppa2505}. The purpose of this note is to provide a remark that this holds for \emph{any} smooth projective surface, following essentially the same argument as in \cite{Stoppa2505}.

To make this precise, we first recall the formulation of this asymptotic stability condition. Following \cite{Stoppa2505,Stoppa2508}, we make the following definition:

\begin{defn}
Let $X$ be a smooth complex projective surface, and let $\omega\in\Amp(X)$ be an ample divisor class. We say a line bundle $L$ on $X$ is \emph{stable in the large scaling limit with respect to $\omega$} if there exists $k_0>0$ such that $L^{\otimes k}$ is $\sigma_{k\omega}$-stable for all $k\geq k_0$. (The stability condition $\sigma_{k\omega}$ on $\Db(X)$ is briefly reviewed in Section~\ref{subsec:Stability}).
\end{defn}

To understand the analytic side of this comparison, we rely on an explicit numerical criterion for the solvability of the dHYM equation for line bundles on surfaces given in \cite[Proposition~8.5]{CollinsJacobYau}:
Let $L$ be a line bundle on a compact K\"ahler surface $(X,\omega)$. A solution to the dHYM equation for $L$ exists if and only if, for every curve $C\subseteq X$, we have
$$
\operatorname{Im}\left(\frac{Z_C(L)}{Z_X(L)}\right) \coloneqq \operatorname{Im}\left(\frac{C.\left(-c_1(L)+\sqrt{-1}\omega\right)}{\frac{1}{2}\left(\omega+\sqrt{-1}c_1(L)\right)^2}\right)>0.
$$
In this note, we will refer to line bundles $L$ satisfying this numerical condition as being \emph{dHYM-stable with respect to $\omega$}. Relaxing this strict inequality yields the corresponding notion of semistability:

\begin{defn}
Let $L$ be a line bundle on a compact K\"ahler surface $(X,\omega)$.
We say $L$ is \emph{dHYM-semistable with respect to $\omega$} if for every curve $C\subseteq X$, we have
$$
\operatorname{Im}\left(\frac{Z_C(L)}{Z_X(L)}\right)\geq0.
$$
\end{defn}

\noindent Note that for a generic K\"ahler class $\omega$, the notions of dHYM-stability and dHYM-semistability are equivalent (see Remark~\ref{rmk:general-dHYM-stable}).

With these definitions established, we are now ready to state the following.

\begin{thm}
\label{thm:MainThm}
Let $X$ be a smooth complex projective surface, $\omega$ an ample divisor class, and $L$ a line bundle. Then we have the following:
\begin{enumerate}[label=(\alph*)]
    \item\label{item:MainThm(a)} If $L$ is stable in the large scaling limit with respect to $\omega$, then $L$ is dHYM-semistable with respect to $\omega$.
    \item\label{item:MainThm(b)} Assuming a conjecture of Arcara and Miles (see Conjecture~\ref{conj:ArcaraMiles} below), if $L$ is dHYM-semistable with respect to $\omega$, then $L^{\otimes k}$ is $\sigma_{k\omega}$-stable for \emph{all} $k\geq1$. (Thus, $L$ is not only stable in the large scaling limit, but stable under \emph{all} scalings of $\omega$.)
\end{enumerate}
\end{thm}

These relationships can be summarized by the following diagram of implications:

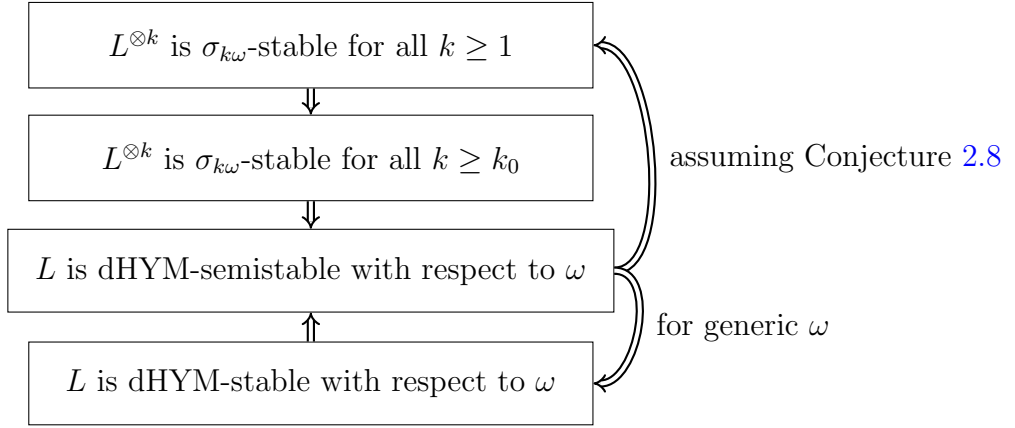
\begin{center}
\begin{tikzpicture}[
    box/.style={
        draw, 
        rectangle, 
        minimum width=7.5cm, 
        minimum height=0.8cm, 
        inner sep=2ex
    },
    implies/.style={
        -Implies, 
        double, 
        double equal sign distance, 
        thick
    }
]

\node[box] (A) at (0, 4.5) {$L^{\otimes k}$ is $\sigma_{k\omega}$-stable for all $k\geq1$};
\node[box] (B) at (0, 3.0) {$L^{\otimes k}$ is $\sigma_{k\omega}$-stable for all $k\geq k_0$};
\node[box] (C) at (0, 1.5) {$L$ is dHYM-semistable with respect to $\omega$};
\node[box] (D) at (0, 0.0) {$L$ is dHYM-stable with respect to $\omega$};

\draw[implies] (A) -- (B);
\draw[implies] (B) -- (C);
\draw[implies] (D) -- (C); 


\draw[implies] (C.east) to[out=0, in=0, looseness=.7] 
    node[right, xshift=1mm] {assuming Conjecture~\ref{conj:ArcaraMiles}} (A.east);

\draw[implies] (C.east) to[out=0, in=0, looseness=1] 
    node[right, xshift=1mm] {for generic $\omega$} (D.east);

\path (C.west) +(-3cm, 0);

\end{tikzpicture}
\end{center}

Note that the genericity assumption appears unavoidable; there exist examples where $L$ is stable under any scaling but fails to be dHYM-stable (see Example~\ref{eg:scalingstable-notdHYMstable}).

\begin{rmk}
In the context of Bridgeland stability conditions, it is customary to consider not only an ample class $\omega$, but a \emph{complexified} ample class $B+i\omega\in \NS(X)_\bR+\sqrt{-1}\Amp(X)$, to which one can associate a Bridgeland stability condition $\sigma_{B,\omega}$. Similarly, for the dHYM equation of line bundles, there is a natural way to incorporate the ``$B$-field'' into the equation, yielding what we refer to as the \emph{$B$-twisted dHYM equation}. We will prove the twisted version of Theorem~\ref{thm:MainThm} in the subsequent section. One can easily recover Theorem~\ref{thm:MainThm} by setting $B=0$. We state the untwisted version in the introduction, as it is the most familiar formulation encountered in the literature.
\end{rmk}

\subsection*{Acknowledgements}
The author would like to thank Professor~Jacopo~Stoppa for enlightening discussions, and Professor~Shing-Tung~Yau for his constant encouragement and support.

\section{Deformed Hermitian--Yang--Mills equations and large scaling limits.}

\subsection{Deformed Hermitian--Yang--Mills equations.}

Let $(X,\omega)$ be a compact K\"ahler manifold of complex dimension $n$, and let $A\in H^{1,1}(X,\bR)$ be a real divisor class. 
The deformed Hermitian--Yang--Mills (dHYM) equation, as studied in \cite{CollinsJacobYau}, asks whether there exists a smooth, closed $(1,1)$-form $\alpha$ such that $[\alpha]=A$ and 
$$
\operatorname{Im}(\omega+\sqrt{-1}\alpha)^n=\tan(\Theta)\operatorname{Re}(\omega+\sqrt{-1}\alpha)^n,
$$
where $\Theta$ is a topological constant determined by $[\omega]$ and $A$.
Note that when $A=c_1(L)$ for a holomorphic line bundle $L\to X$, a solution to the dHYM equation yields a dHYM metric on $L$, an object that plays an important role in SYZ mirror symmetry \cite{LYZ}.

One can incorporate a background $B$-field into this geometric setup, which leads to the following twisted version of the equation:

\begin{defn}
Let $(X,\omega)$ be a compact K\"ahler manifold, $L$ a line bundle on $X$, and $B\in H^{1,1}(X,\bR)$ a real divisor class.
We say a smooth, closed $(1,1)$-form $\alpha$ is a solution to the \emph{$B$-twisted dHYM equation for $L$} if
$$
\begin{cases}
    [\alpha]=c_1(L)-B, \\
    \operatorname{Im}(\omega+\sqrt{-1}\alpha)^n=\tan(\Theta)\operatorname{Re}(\omega+\sqrt{-1}\alpha)^n.
\end{cases}
$$
\end{defn}

In complex dimension two, the existence of solutions to this equation is completely characterized by an explicit numerical criterion.

\begin{thm}[{\cite[Proposition 8.5]{CollinsJacobYau}}]
Let $L$ be a line bundle on a compact K\"ahler surface $(X,\omega)$. 
A solution to the $B$-twisted dHYM equation for $L$ exists if and only if, for every curve $C\subseteq X$, we have
$$
\operatorname{Im}\left(\frac{Z_C}{Z_X}\right)\coloneqq\operatorname{Im}\left(\frac{C.\left(-(c_1(L)-B)+\sqrt{-1}\omega\right)}{\frac{1}{2}\left(\omega+\sqrt{-1}(c_1(L)-B)\right)^2}\right)>0.
$$
In this case, we say $L$ is \emph{$B$-twisted dHYM-stable with respect to $\omega$}.
\end{thm}

\begin{rmk}
\label{rmk:dHYM}
Let $L$ be a line bundle on a surface $(X, \omega)$.
\begin{enumerate}[label=(\alph*)]
    \item\label{item:dHYM-slope-0-stable} If $\omega.(c_1(L)-B)=0$, then the above positivity condition is automatically satisfied for any curve $C\subseteq X$; therefore, a solution to the $B$-twisted dHYM equation for $L$ always exists in this case.
    \item Expanding the intersection products, the above phase inequality can be algebraically rewritten as:
$$
(C.\omega)\left(\omega^2-(c_1(L)-B)^2\right)+2(C.(c_1(L)-B))(\omega.(c_1(L)-B))>0.
$$
\end{enumerate}
\end{rmk}

Weakening the strict inequality leads to the corresponding notion of semistability.

\begin{defn}
We say $L$ is \emph{$B$-twisted dHYM-semistable with respect to $\omega$} if for every curve $C\subseteq X$, we have
$$
\operatorname{Im}\left(\frac{Z_C}{Z_X}\right)\geq0.
$$
Equivalently, this can be written as
$$
(C.\omega)\left(\omega^2-(c_1(L)-B)^2\right)+2(C.(c_1(L)-B))(\omega.(c_1(L)-B))\geq0.
$$
\end{defn}

\begin{rmk}\label{rmk:general-dHYM-stable}
For generic classes $B$ and $\omega$, the notions of dHYM-semistability and dHYM-stability coincide. To see this, observe that each pair of classes $[C],c_1(L)\in H^{1,1}(X,\bZ)$ defines a proper analytic subvariety of the complexified K\"ahler cone $H^{1,1}(X,\bR)+\sqrt{-1}\operatorname{K\ddot{a}h}(X)$ given by the equation 
$$
(C.\omega)\left(\omega^2-(c_1(L)-B)^2\right)+2(C.(c_1(L)-B))(\omega.(c_1(L)-B))=0.
$$
Provided that the pair $(B,\omega)$ avoids this countable union of analytic subvarieties, the equivalence holds.
\end{rmk}

\subsection{Geometric stability conditions on surfaces.}
\label{subsec:Stability}
We refer the reader to the original papers \cite{BriStab,BriK3,ArcaraBertram} for the general definition of Bridgeland stability conditions, as well as the construction of geometric stability conditions on smooth projective surfaces. Here, we recall only the essential statements required for our subsequent discussion.

Let $X$ be a smooth complex projective surface, and let $B,\omega\in\NS(X)\otimes\bR$ with $\omega$ ample. 
Consider the slope function
$$
\mu_\omega(E)\coloneqq\frac{\omega.c_1(E)}{\rank(E)},
$$
defined for coherent sheaves $E\in\Coh(X)$.
Using this slope, one defines the tilted heart 
$$
\sA_{B,\omega}=\{E\in\Db(X) \mid H^0(E)\in\sT_{B,\omega},\,H^{-1}(E)\in\sF_{B,\omega},\,H^i(E)=0\text{ for }i\neq0,-1\},
$$
where 
\begin{itemize}
    \item $\sT_{B,\omega}\subseteq\Coh(X)$ is the full subcategory closed under extensions generated by the torsion sheaves and $\mu_\omega$-stable sheaves with slope $\mu_\omega(E)>B.\omega$.
    \item $\sF_{B,\omega}\subseteq\Coh(X)$ is the full subcategory closed under extensions generated by $\mu_\omega$-stable sheaves with slope $\mu_\omega(E)\leq B.\omega$.
\end{itemize}
Note that $\sA_{B,\omega}$ is the heart of a bounded t-structure on $\Db(X)$, and is therefore an abelian category.

\begin{thm}[\cite{ArcaraBertram,BriK3}]
Let $X$ be a smooth complex projective surface, and let $B,\omega\in\NS(X)\otimes\bR$ with $\omega$ ample. Then
$$
\sigma_{B,\omega}\coloneqq(Z_{B,\omega},\sA_{B,\omega})
$$
is a stability condition on $\Db(X)$, where the central charge $Z_{B,\omega}$ is given by
$$
Z_{B,\omega}(E)=-\int_Xe^{-(B+\sqrt{-1}\omega)}\ch(E).
$$
\end{thm}

As this note primarily concerns the stability of line bundles, their behavior with respect to this tilted heart is of particular interest.

\begin{rmk}
\label{rmk:BSC}
Let $L$ be a line bundle on $X$. Then:
\begin{enumerate}[label=(\alph*)]
    \item\label{item:BSC-slope>0} $L\in\sA_{B,\omega}$ if and only if $\omega.(c_1(L)-B)>0$.
    \item\label{item:BSC-slope<=0} $L[1]\in\sA_{B,\omega}$ if and only if $\omega.(c_1(L)-B)\leq0$.
    \item\label{item:BSC-slope=0} If $\omega.(c_1(L)-B)=0$, then $L$ is $\sigma_{B,\omega}$-stable.
\end{enumerate}
The first two statements follow directly from the construction of the tilted heart. The third statement follows from the characterization of objects in $\sP_{B,\omega}(1)$ (cf.~\cite[Lemma~10.1(b)]{BriK3}).
\end{rmk}

The following conjecture characterizes the precise mechanisms by which a line bundle can fail to be Bridgeland stable.

\begin{conj}[{\cite[Conjecture~1]{ArcaraMiles}}]
\label{conj:ArcaraMiles}
Let $L$ be a line bundle on $X$.
\begin{enumerate}[label=(\alph*)]
    \item\label{item:ArcaraMiles(a)} Suppose $L\in\sA_{B,\omega}$. Then $L$ is not $\sigma_{B,\omega}$-stable if and only if there exists a curve $C$ of negative self-intersection such that $L(-C)$ is a subobject of $L$ in $\sA_{B,\omega}$ with $0<\Arg Z_{B,\omega}(L)\leq\Arg Z_{B,\omega}(L(-C))\leq\pi$.
    \item\label{item:ArcaraMiles(b)} Suppose $L[1]\in\sA_{B,\omega}$. Then $L$ is not $\sigma_{B,\omega}$-stable if and only if there exists a curve $C$ of negative self-intersection such that $L(C)|_C$ is a subobject of $L[1]$ in $\sA_{B,\omega}$ with $0<\Arg Z_{B,\omega}(L[1])\leq\Arg Z_{B,\omega}(L(C)|_C)\leq\pi$.
\end{enumerate}
\end{conj}

This conjecture has been established in several cases, including for surfaces $X$ with no curves of negative self-intersection, for surfaces with $\rank\NS(X)=2$ containing a unique irreducible curve of negative self-intersection \cite[Theorem~1.1]{ArcaraMiles}, and for del Pezzo surfaces with $\rank\NS(X)=3$ \cite{MY25}.

\subsection{Relation between dHYM stability and Bridgeland stability.}

With the algebraic and analytic frameworks in place, we are now ready to establish the connection between them. We prove the following twisted version of Theorem~\ref{thm:MainThm}:

\begin{thm}
Let $X$ be a smooth complex projective surface, $B,\omega\in\NS(X)\otimes\bR$ with $\omega$ ample, and $L$ a line bundle. Then we have the following:
\begin{enumerate}[label=(\alph*)]
    \item\label{item:twisted(a)} If there exists $k_0>0$ such that $L^{\otimes k}$ is $\sigma_{kB,k\omega}$-stable for all $k\geq k_0$, then $L$ is $B$-twisted dHYM-semistable with respect to $\omega$.
    \item\label{item:twisted(b)} Assuming Conjecture~\ref{conj:ArcaraMiles} holds for $X$, if $L$ is $B$-twisted dHYM-semistable with respect to $\omega$, then $L^{\otimes k}$ is $\sigma_{kB, k\omega}$-stable for \emph{all} $k\geq1$. 
\end{enumerate}
\end{thm}

\begin{proof}[Proof of \ref{item:twisted(a)}]
For every curve $C\subseteq X$, we wish to verify the inequality
$$
(C.\omega)\left(\omega^2-(c_1(L)-B)^2\right)+2(C.(c_1(L)-B))(\omega.(c_1(L)-B))\geq0.
$$
To do this, we separate the analysis into cases depending on the slope of $L$.

\medskip
\noindent\textbf{Case 1: $\omega.(c_1(L)-B)=0$.}
In this case, $L$ is always $B$-twisted dHYM-stable by Remark~\ref{rmk:dHYM}\ref{item:dHYM-slope-0-stable}.

\medskip
\noindent\textbf{Case 2: $\omega.(c_1(L)-B)>0$.}
By Remark~\ref{rmk:BSC}\ref{item:BSC-slope>0}, we have $L\in\sA_{B,\omega}$.
Similarly, $L^{\otimes k}\in\sA_{kB,k\omega}$ for all $k\geq1$ since
$$
k\omega.(c_1(L^{\otimes k})-kB)=k^2\omega.(c_1(L)-B)>0.
$$
For any curve $C\subseteq X$, consider the line bundle $L^{\otimes k}(-C)$. Observe that
$$
k\omega.(c_1(L^{\otimes k}(-C))-kB)=
k\omega.(kc_1(L)-[C]-kB)>0
$$
for $k$ sufficiently large.
Therefore, $L^{\otimes k}(-C)\in\sA_{kB,k\omega}$ for large $k$.
Because $L^{\otimes k}|_C$ is a torsion sheaf (and therefore lies in the heart), this yields a short exact sequence
$$
0\to L^{\otimes k}(-C) \to L^{\otimes k} \to L^{\otimes k}|_C \to 0
$$
in $\sA_{kB,k\omega}$ for large $k$.

By assumption, $L^{\otimes k}$ is $\sigma_{kB,k\omega}$-stable for large $k$, which enforces the phase condition:
$$
0<\Arg Z_{kB,k\omega}(L^{\otimes k}(-C))<\Arg Z_{kB,k\omega}(L^{\otimes k})<\pi.
$$
A direct computation shows that this phase inequality is equivalent to
$$
(C.\omega)\left(\omega^2-(c_1(L)-B)^2\right)+\left(2C.(c_1(L)-B)-\frac{C^2}{k}\right)(\omega.(c_1(L)-B))>0.
$$
Since this inequality holds for any sufficiently large $k$, taking the limit as $k\to\infty$ yields
$$
(C.\omega)\left(\omega^2-(c_1(L)-B)^2\right)+2(C.(c_1(L)-B))(\omega.(c_1(L)-B))\geq0,
$$
which is precisely the condition for dHYM-semistability.

\medskip
\noindent\textbf{Case 3: $\omega.(c_1(L)-B)<0$.}
In this case, we have $L^{\otimes k}[1]\in\sA_{kB,k\omega}$ for all $k\geq1$ by Remark~\ref{rmk:BSC}\ref{item:BSC-slope<=0}.
Fixing a curve $C\subseteq X$ and applying a similar argument, we find that for large $k$, $L^{\otimes k}(C)[1]\in\sA_{kB,k\omega}$. This yields the following short exact sequence in the heart $\sA_{kB,k\omega}$:
$$
0 \to L^{\otimes k}(C)|_C \to L^{\otimes k}[1] \to L^{\otimes k}(C)[1] \to 0.
$$
By assumption, $L^{\otimes k}[1]$ is $\sigma_{kB,k\omega}$-stable for large $k$, which requires $\Arg Z_{kB,k\omega}(L^{\otimes k}[1])<\Arg Z_{kB,k\omega}(L^{\otimes k}(C)[1])$. This is equivalent to
$$
(C.\omega)\left(\omega^2-(c_1(L)-B)^2\right)+\left(2C.(c_1(L)-B)+\frac{C^2}{k}\right)(\omega.(c_1(L)-B))>0.
$$
Taking the limit as $k\to\infty$, this again implies dHYM-semistability.
\end{proof}

\begin{proof}[Proof of \ref{item:twisted(b)}]
As before, we separate the analysis into cases depending on the slope of $L$.

\medskip
\noindent\textbf{Case 1: $\omega.(c_1(L)-B)=0$.}
In this boundary case, $L^{\otimes k}$ is immediately $\sigma_{kB,k\omega}$-stable for all $k\geq1$ by Remark~\ref{rmk:BSC}\ref{item:BSC-slope=0}.

\medskip
\noindent\textbf{Case 2: $\omega.(c_1(L)-B)>0$.}
Here, we wish to verify the $\sigma_{kB,k\omega}$-stability of $L^{\otimes k}$.
Assuming Conjecture~\ref{conj:ArcaraMiles}\ref{item:ArcaraMiles(a)}, it suffices to check the phase inequality for all curves $C\subseteq X$ with $C^2<0$ such that $L^{\otimes k}(-C)$ is a subobject of $L^{\otimes k}$ in the heart.
By the algebraic expansion from the previous proof, this is equivalent to showing that
$$
(C.\omega)\left(\omega^2-(c_1(L)-B)^2\right)+\left(2C.(c_1(L)-B)-\frac{C^2}{k}\right)(\omega.(c_1(L)-B))>0.
$$
Since $C^2<0$ and $\omega.(c_1(L)-B)>0$, the added term $-\frac{C^2}{k}(\omega.(c_1(L)-B))$ is strictly positive. Therefore, this required inequality is an immediate consequence of the given dHYM-semistability condition.

\medskip
\noindent\textbf{Case 3: $\omega.(c_1(L)-B)<0$.}
Finally, assuming Conjecture~\ref{conj:ArcaraMiles}\ref{item:ArcaraMiles(b)}, it suffices to verify that for all curves $C$ with $C^2<0$, we have
$$
(C.\omega)\left(\omega^2-(c_1(L)-B)^2\right)+\left(2C.(c_1(L)-B)+\frac{C^2}{k}\right)(\omega.(c_1(L)-B))>0.
$$
Since $C^2<0$ and $\omega.(c_1(L)-B)<0$, the additional term $+\frac{C^2}{k}(\omega.(c_1(L)-B))$ is, once again, strictly positive. Thus, the required strict inequality follows directly from dHYM-semistability.
\end{proof}

\subsection{Examples.}
We conclude this note with two examples illustrating the subtleties of scaling and strict stability.

\begin{eg}
\label{eg:not-scaling-invariant}
Let us recall an example from \cite[Page~31]{CollinsShi}. Let $X=\operatorname{Bl}_{\text{pt}}(\bP^2)$ be the blowup of $\bP^2$ at a point. Denote the pullback of the hyperplane class from $\bP^2$ by $H$, and let $E$ be the exceptional divisor. Then
$$
\omega=\frac{1}{\sqrt3}(2H-E)
$$
is an ample divisor class. Consider a line bundle $L$ with $c_1(L)=2H$.

As noted in \cite[Page~31]{CollinsShi} (relying on \cite{ArcaraMiles}), $L$ is $\sigma_{0,\omega}$-stable but fails to be dHYM-semistable with respect to $\omega$. This provides a perfect illustration of the \emph{non-scaling-invariance} of Bridgeland stability conditions. 

Because Conjecture~\ref{conj:ArcaraMiles} holds for this surface \cite[Theorem~1.1]{ArcaraMiles}, Theorem~\ref{thm:MainThm} implies that $L$ would be dHYM-semistable with respect to $\omega$ if and only if $L^{\otimes k}$ were $\sigma_{0,k\omega}$-stable for all $k\geq1$. However, while $L$ itself (the $k=1$ case) is $\sigma_{0,\omega}$-stable, its higher tensor powers $L^{\otimes k}$ are not $\sigma_{0,k\omega}$-stable for any $k\geq2$. To see this, note that both $L^{\otimes k}$ and $L^{\otimes k}(-E)$ lie in the heart $\sA_{0,k\omega}$. Evaluating the stability phase condition on the exceptional curve $E$, we obtain:
$$
(E.\omega)\left(\omega^2-c_1(L)^2\right)+\left(2E.c_1(L)-\frac{E^2}{k}\right)(\omega.c_1(L))=
\frac{1}{\sqrt{3}}\left(-3+\frac{4}{k}\right)<0 \quad \text{for} \quad k\geq2.
$$
\end{eg}

\begin{eg}
\label{eg:scalingstable-notdHYMstable}
We again consider $X=\operatorname{Bl}_{\text{pt}}(\bP^2)$ with the same ample class $\omega$. We claim that the line bundle $L$ with $c_1(L)=H$ satisfies the following properties:
\begin{enumerate}[label=(\alph*)]
    \item $L$ is not dHYM-stable with respect to $\omega$, yet
    \item $L^{\otimes k}$ is $\sigma_{0,k\omega}$-stable for all $k\geq 1$.
\end{enumerate}

To verify (a), we evaluate the dHYM-stability criterion on the exceptional curve $E$:
$$
(E.\omega)\left(\omega^2-c_1(L)^2\right)+2(E.c_1(L))(\omega.c_1(L))=\frac{1}{\sqrt{3}}(0)+2(0)\left(\frac{2}{\sqrt{3}}\right)=0.
$$
Because this is not strictly positive, $L$ is not dHYM-stable (although it is dHYM-semistable).

To verify (b), we rely again on the fact that Arcara and Miles proved their conjecture in this setting \cite[Theorem~1.1]{ArcaraMiles}. Therefore, to show that $L^{\otimes k}$ is $\sigma_{0,k\omega}$-stable for all $k\geq 1$, it suffices to demonstrate that the strict phase inequality holds for every curve $C$ with $C^2<0$:
$$
(C.\omega)\left(\omega^2-c_1(L)^2\right)+\left(2C.c_1(L)-\frac{C^2}{k}\right)(\omega.c_1(L))>0.
$$
Since $\omega^2-c_1(L)^2=0$ and $\omega.c_1(L)>0$, this reduces to showing that
$$
2C.H-\frac{C^2}{k}>0.
$$
Given that $C^2<0$, the second term $-\frac{C^2}{k}$ is strictly positive, so it simply suffices to check that $C.H\geq0$. This follows from the elementary fact that the classes of effective curves with negative self-intersection on $X$ are of the form
$$
aH+bE \qquad \text{ for some } \qquad b>a\geq0.
$$
Thus,
$$
C.H=(aH+bE).H=a\geq0,
$$
which confirms that $L^{\otimes k}$ is $\sigma_{0,k\omega}$-stable for all $k\geq 1$.
\end{eg}

\bigskip
\bibliography{ref}
\bibliographystyle{alpha}

\ \\

\footnotesize{
\noindent Yu-Wei Fan \\
\textsc{Center for Mathematics and Interdisciplinary Sciences, Fudan University, \\ Shanghai 200433, China}\\
\textsc{Shanghai Institute for Mathematics and Interdisciplinary Sciences (SIMIS), \\  Shanghai 200433, China}\\
\texttt{yuweifanx@gmail.com}
}

\end{document}